\documentclass{amsart}
\usepackage{amsfonts,amssymb,amsmath,amsthm}
\usepackage{url}
\usepackage{enumerate}

\usepackage{graphicx}

\newtheorem{theorem}{Theorem}[section]

\newtheorem{lemma}[theorem]{Lemma}
\newtheorem{claim}[theorem]{Claim}
\newtheorem{subclaim}{Subclaim}

\theoremstyle{definition}
\newtheorem{remark}[theorem]{Remark}

\begin{document}

\title[Toroidal Seifert fibered surgeries on Montesinos knots]
{Toroidal Seifert fibered surgeries\\ on Montesinos knots}

\dedicatory{Dedicated to Professor Kunio Murasugi on the occasion of his 80th birthday}

\author{Kazuhiro Ichihara}
\address{Department of Mathematics, College of Humanities and Sciences, Nihon University, 
3-25-40 Sakurajosui, Setagaya-ku, Tokyo 156-8550, Japan}
\email{ichihara@math.chs.nihon-u.ac.jp}
\urladdr{http://www.math.chs.nihon-u.ac.jp/~ichihara/index.html}
\thanks{The first author is partially supported by
Grant-in-Aid for Young Scientists (B), No.~20740039,
Ministry of Education, Culture, Sports, Science and Technology, Japan.}

\author{In Dae Jong}
\address{Osaka City University Advanced Mathematical Institute, 3-3-138, Sugimoto, Sumiyoshi-ku, Osaka, 558-8585, Japan}
\email{jong@sci.osaka-cu.ac.jp}
\urladdr{http://jong.web.fc2.com/index.html}
\thanks{The second author is partially supported by Grant-in-Aid for Research Activity Start-up, No.~22840037, Japan Society for the Promotion of Science.}


\begin{abstract}
We show that if a Montesinos knot admits 
a Dehn surgery yielding a toroidal Seifert fibered $3$-manifold, 
then the knot is the trefoil knot and the surgery slope is $0$. 
\end{abstract}

\date{\today}

\maketitle

\section{Introduction}

Let $K$ be a knot in the $3$-sphere $S^3$ and $E(K)$ the exterior of $K$. 
For a slope $\gamma$ 
on the boundary of $E(K)$, 
we denote by $K(\gamma)$ 
the manifold obtained by 
the \textit{Dehn surgery on $K$ along a slope $\gamma$}, 
i.e., $K(\gamma)$ is obtained by 
gluing a solid torus $V$ to $E(K)$ 
so that a simple closed curve representing 
$\gamma$ bounds a disk in $V$. 
We call such a slope $\gamma$ the {\it surgery slope}. 
It is well-known that a slope on the boundary torus $\partial E(K)$ 
is parameterized by an element of $\mathbb{Q} \cup \{ 1/0 \}$ 
by using the standard meridian-longitude system for $K$. 
Thus, when a slope $\gamma$ corresponds to $r \in \mathbb{Q} \cup \{1/0\}$, 
we call the Dehn surgery along $\gamma$ 
the \textit{$r$-surgery} for brevity, and 
denote $K(\gamma)$ by $K(r)$.
Since $K(1/0)$ is homeomorphic to $S^3$ again, 
$1/0$-surgery is called the \textit{trivial} surgery. 
See \cite{KawauchiBook}, \cite{RolfsenBook} for basic references.

The Hyperbolic Dehn Surgery Theorem, 
established by Thurston~\cite[Theorem 5.8.2]{Thurston1978}, 
says that each hyperbolic knot admits 
only finitely many Dehn surgeries yielding non-hyperbolic manifolds. 
Here a knot is called \textit{hyperbolic} 
if its complement admits a complete hyperbolic structure of finite volume. 
Thereby a Dehn surgery on a hyperbolic knot 
is called \textit{exceptional} if it yields a non-hyperbolic manifold. 
As a consequence of the Geometrization Conjecture, raised by Thurston~\cite[section 6, question 1]{Thurston1982}, and established by Perelman's works~\cite{Perelman2002, Perelman2003, Perelman2003a}, exceptional surgeries are classified into the three types: 
a Seifert fibered surgery, a toroidal surgery, and a reducible surgery. 
Here a Dehn surgery is called \textit{Seifert fibered} / \textit{toroidal} / \textit{reducible} 
if it yields a Seifert fibered / toroidal / reducible manifold. 
We refer the reader to \cite{Boyer2002} for a survey. 

We here note that the classification is not exclusive, 
that is, there exist Seifert fibered 3-manifolds which are both toroidal and reducible. 
However, 
hyperbolic knots in $S^3$ are conjectured to admit no reducible surgeries. 
This is the well-known, but still open, Cabling Conjecture \cite{Gonzalez-AcunaShort1986}.
Thus we consider in this paper a Dehn surgery on a knot in $S^3$ 
yielding a $3$-manifold which is toroidal and Seifert fibered, 
which we call a \textit{toroidal Seifert fibered surgery}. 

Actually there exist infinitely many hyperbolic knots in $S^3$ 
each of which admits a toroidal Seifert fibered surgery. 
These were found by Eudave-Mu\~{n}oz \cite[Proposition 4.5(1) and (3)]{Eudave-Munoz2002}, 
and Gordon and Luecke \cite{GordonLuecke1999a} independently. 
On the other hand, Motegi~\cite{Motegi2003} studied 
toroidal Seifert fibered surgeries on symmetric knots, 
and gave several restrictions on the existence of such surgeries. 
In particular, he showed that 
only the trefoil knot admits a toroidal Seifert fibered surgery among two-bridge knots \cite[Corollary 1.6]{Motegi2003}. 

Extending this result, in this paper, we show the following: 

\begin{theorem}\label{thm:main}
Let $K$ be a Montesinos knot in the $3$-sphere. 
If $K$ admits a toroidal Seifert fibered surgery, 
then $K$ is the trefoil knot and the surgery slope is $0$. 
\end{theorem}

A \textit{Montesinos knot} of type $(R_1, \dots, R_l)$, 
denoted by $M(R_1, \dots, R_l)$, 
is defined as 
a knot admitting a diagram obtained by 
putting rational tangles $R_1, \dots, R_l$ 
together in a circle 
(see Figure~\ref{fig:DefMonte}). 
For brevity, here and in the sequel, 
we abuse $R_i$ to denote 
a rational number (or, an irreducible fraction) 
or the corresponding rational tangle depending on the context. 
The minimal number of such rational tangles is called 
the \textit{length} of a Montesinos knot. 
In particular, a Montesinos knot $K$ is called 
a \textit{pretzel knot} of type $(a_1, \dots,a_l)$, 
denoted by $P(a_1, \dots, a_l)$, 
if the rational tangles in $K$ are of the form $1/a_1, \dots, 1/a_l$. 

\begin{figure}[htb]	
\includegraphics[width=0.3\textwidth]{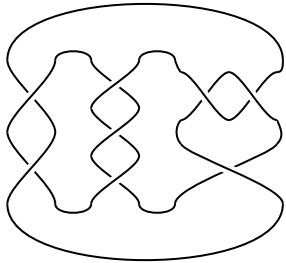}
\caption{$M(1/2 , 1/3, -2/3)$}\label{fig:DefMonte}
\end{figure}

Exceptional surgeries on Montesinos knots are extensively studied. 
See \cite{IchiharaJong2009}, \cite{Wu1996}, \cite{Wu1998a}, \cite{Wu2006} for example.

On toroidal Seifert fibered surgeries, 
in our forthcoming paper \cite{IchiharaJongMizushima2010}, we will also show that 
among prime alternating knots, only the trefoil knot admits a toroidal Seifert fibered surgery.

\section{Proof of Theorem~\ref{thm:main}}

To prove Theorem~\ref{thm:main}, 
we will show the following two lemmas in Sections 3 and 4 respectively. 

\setcounter{theorem}{0}
\setcounter{section}{3}
\begin{lemma}\label{lem:pretzel}
For the pretzel knot $K_{m,n}=P(2m, 2n + 1, -2n-1)$ 
with integers $m \ne 0$ and $n \ge 1$, 
the surgered manifold $K_{m,n}(0)$ is not Seifert fibered. 
\end{lemma}

\setcounter{theorem}{0}
\setcounter{section}{4}
\begin{lemma}\label{lem:non-pretzel}
For the Montesinos knot $K_\pm=M(-1/2, 1/3, 1/(6 \pm 1/2) )$, 
the surgered manifold $K_\pm(0)$ is not Seifert fibered. 
\end{lemma}

\setcounter{section}{2}
\setcounter{theorem}{0}

Assuming these lemmas, we prove Theorem~\ref{thm:main}.

\begin{proof}[Proof of Theorem~\ref{thm:main}]
Let $K$ be a Montesinos knot admitting a toroidal Seifert fibered surgery. 
That is, we suppose that $K(r)$ is a toroidal Seifert fibered manifold for some $r \in \mathbb{Q}$. 
Then the length of $K$ must be less than four 
since Montesinos knots of length at least four admit no exceptional surgeries \cite{Wu1996}. 
On the other hand, 
if the length of $K$ is less than three, meaning that $K$ is a two-bridge knot, 
then $K$ is the trefoil knot and $r=0$, i.e., the surgery slope is longitudinal~\cite[Corollary 1.6]{Motegi2003}. 

Thus, in the following, we suppose that $K$ is of length three. 
In addition, if $K$ is non-hyperbolic, then $K$ is either $P(-2,3,3)$ or $P(-2,3,5)$, 
which are actually the torus knots of type $(3,4)$ or of type $(3,5)$. 
This was originally shown by Oertel~\cite[Corollary 5]{Oertel1984} 
together with the result in the unpublished monograph by Bonahon and Siebenmann~\cite{BonahonSiebenmann1979-85, BonahonSiebenmann2010}. 
Dehn surgeries on torus knots are completely classified~\cite{Moser1971}, and 
the two torus knots admit no toroidal Seifert fibered surgeries. 
Thus, from now on, we further assume that $K$ is hyperbolic. 
Then we have the following: 

\begin{claim}\label{clm:cand}
The knot $K$ must be one of the following: 
\begin{enumerate}
\item[{\rm (I)}] $P(2m, 2n + 1, -2n-1)$ for integers $m \ne 0$ and $n \ge 1$.
\item[{\rm (II)}] $M(-1/2, 1/3, 1/(6 \pm 1/2) )$. 
\end{enumerate}
Furthermore the slope $r$ must be $0$. 
\end{claim}
\begin{proof}
Since the exterior of a Montesinos knot of length three contains no closed essential surface~\cite[Corollary 4]{Oertel1984}, 
if $K(r)$ is toroidal, 
then $r$ must be a boundary slope of an essential surface. 
Furthermore, if $K(r)$ is Seifert fibered, 
then by \cite[Proposition 1]{IchiharaMotegiSong2008}, 
it follows that $K$ is fibered and $r=0$. 

On the other hand, 
toroidal surgeries on hyperbolic Montesinos knots of length three are 
completely listed by Wu~\cite{Wu2006}. 
Thus, the candidates for toroidal Seifert fibered surgeries on Montesinos knots 
are $0$-surgeries on fibered knots contained in Wu's list. 
The $0$-surgeries contained in Wu's list are the following: 
\begin{enumerate}
\item[(i)] The $0$-surgery on $P(2p_1 -1, 2p_2 -1, 2p_3 -1)$ for integers $p_i \ne 0,1$. 
\item[(ii)] The $0$-surgery on $P(2m, 2n +1, -2n - 1)$ for integers $m \ne 0$ and $n \ge 1$. 
\item[(iii)] The $0$-surgery on $M(-1/2, 1/3, 1/(6 + 1/2s))$ for an integer $s \ne 0$.
\end{enumerate}

Notice that the first class of pretzel knots are of genus one. 
It is known that a genus one fibered knot is either the trefoil knot or the figure-eight knot \cite{Gonzalez-Acuna1970}. 
Therefore $P(2p_1 -1, 2p_2 -1, 2p_3 -1)$ for integers $p_i \ne 0,1$ is not fibered. 

On the third case, calculating the Alexander polynomial, we have
$$\Delta_{M(-1/2, 1/3, 1/(6 + 1/2s))}(t) = -s(t^2+t^{-2}) +2s + 1.$$
Since the Alexander polynomials of fibered knots have to be monic 
(see \cite{RolfsenBook} for example), 
if the knot $M(-1/2, 1/3, 1/(6 + 1/2s))$ is fibered, then $s = \pm1$. 
This completes the proof of Claim~\ref{clm:cand}. 
\end{proof}

However, the $0$-surgeries on the knots listed in Claim~\ref{clm:cand} are not Seifert fibered surgeries by Lemmas~\ref{lem:pretzel} and \ref{lem:non-pretzel} respectively. 
Therefore a Montesinos knot of length three admits no toroidal Seifert fibered surgery, 
and this completes the proof of Theorem~\ref{thm:main}. 
\end{proof}

\begin{remark}\label{rem:fibered}
By using Gabai's arguments~\cite{Gabai1986a}, we see that each of the knots listed in Claim~\ref{clm:cand} is fibered. 
\end{remark}

Let $\mathbf{B}_l (S^2)$ be the $l$-string braid group on the $2$-sphere. 
We denote by $\sigma_1, \dots, \sigma_{l-1}$ the standard generators of $\mathbf{B}_{l}(S^2)$. 
See \cite{MurasugiKurupitaBook} for details about braids on the $2$-sphere. 
The following lemma is used in both Sections~\ref{sec:pretzel} and \ref{sec:Monte}.

\begin{lemma}\label{lem:exp}
Let $\widehat{b}$ and $\widehat{b'}$ be the closed braids for $b \in \mathbf{B}_l (S^2)$ and $b' \in \mathbf{B}_{l'} (S^2)$ respectively. 
If $\widehat{b}$ is isotopic to $\widehat{b'}$ in $S^2 \times S^1$, 
then the numbers of strands of $b$ and $b'$ coincide, 
i.e., $l = l'$ holds. 
Furthermore, with respect to the standard generators, 
the exponent sum of $b$ is congruent to that of $b'$ modulo $2(l-1)$. 
\end{lemma}
\begin{proof}
It is a well-known fact that $\widehat{b}$ is isotopic to $\widehat{b'}$ in $S^2 \times S^1$ 
if and only if $b$ and $b'$ are both $l$-strand braids on the $2$-sphere for some positive integer $l$, 
and $b$ is conjugate to $b'$ in $\mathbf{B}_l (S^2)$. 
See \cite[Section 3]{Murasugi1982} for example. 
Furthermore, by \cite[Chapter 11, Proposition 2.3]{MurasugiKurupitaBook}, 
if $b$ is conjugate to $b'$ in $\mathbf{B}_l(S^2)$, 
then the exponent sum of $b$ is congruent to that of $b'$ modulo $2(l-1)$. 
\end{proof}

\section{On pretzel knots}\label{sec:pretzel}

In this section, we prove the following: 

\begin{lemma}
For the pretzel knot $K_{m,n}=P(2m, 2n + 1, -2n-1)$ 
with integers $m \ne 0$ and $n \ge 1$, 
the surgered manifold $K_{m,n}(0)$ is not Seifert fibered. 
\end{lemma}
\begin{proof}
We note that the knot $K_{m,n} = P(2m,2n+1,-2n-1)$ is the mirror image of the knot $K_{-m,n} = P(-2m,2n+1,-2n-1)$. 
Thus, if $K_{m,n}(0)$ is not Seifert fibered for $m \ge 1$, then $K_{-m,n}(0)$ is also not Seifert fibered.

We first consider the case where $n=1$ as follows. 

\begin{claim}\label{clm:n=1}
The surgered manifold $K_{m,1}(0)$ is not Seifert fibered. 
\end{claim}
\begin{proof}
Applying the Montesinos trick~\cite{Montesinos1975}, the surgered manifold $K_{m,1} (0)$ is homeomorphic to the double branched covering space of $S^3$ branched along the two-component link $\lambda_m$ depicted in Figure~\ref{fig:MonteTrick}. 
In the decomposition of the link $\lambda_m$ as a partial sum of two 2-string tangles, 
the left part can be viewed as a Montesinos tangle or a Montesinos-m tangle in the sense of~\cite{Eudave-Munoz2002}, see Figure~\ref{fig:decomp}. 
Correspondingly the double branched cover of the left tangle is 
either a Seifert fibered space over the disk with two exceptional fibers of indices $2$ 
or a Seifert fibered space over the M\"{o}bius band 
without exceptional fibers. 
In either of these cases, 
it is verified that 
their Seifert fibrations do not match since $m \ne 0$ (see \cite[Appendix A]{BonahonSiebenmann2010}, \cite[Chapter IV]{JacoBook} for example). 
Thus, $K_{m,1}(0)$ is not Seifert fibered. 
\end{proof}


\begin{figure}[htb]
\includegraphics[width=1.0\textwidth]{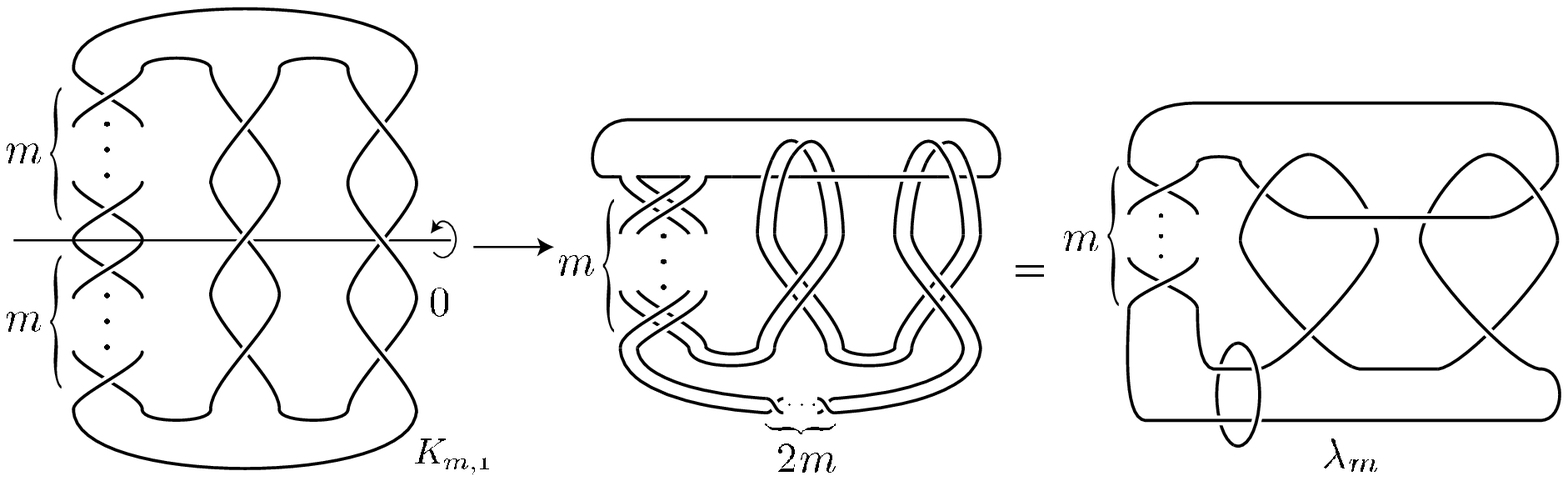}
\caption{}\label{fig:MonteTrick}
\end{figure}
\begin{figure}[htb]
\includegraphics[width=.4\textwidth]{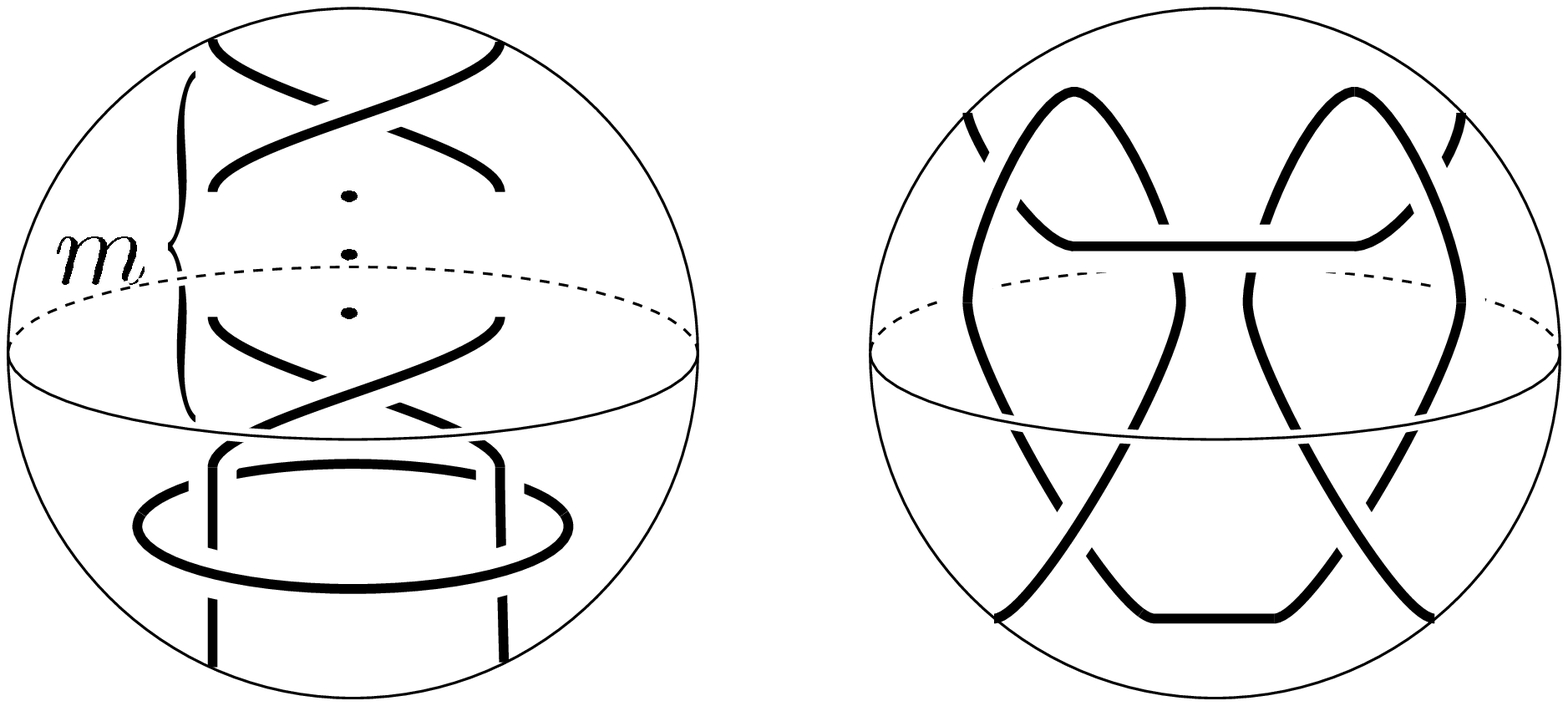}
\caption{}\label{fig:decomp}
\end{figure}

Suppose that $K_{m,n} (0)$ ($n \ge 2$) is Seifert fibered. 
As illustrated in the upper half of Figure~\ref{fig:surgery1}, let $\iota$ be a strong inversion of $K_{m,n}$ with respect to an axis $\alpha$, 
and $F$ a non-orientable spanning surface invariant under $\iota$. 
Let $V$ be the attached solid torus via the $0$-surgery. 
Then there exists a meridian disk $D$ in $V$ such that 
$F \cup D$ gives a Klein bottle $\widehat{F}$ embedded in $K_{m,n}(0)$. 
For a manifold $X$ and a sub-manifold $Y$, we denote by $N(Y)$ a regular neighborhood of $Y$ in $X$ and by $N^\circ(Y)$ the interior of $N(Y)$. 

To consider the remaining cases where $n\ge2$, the following is the key claim. 

\begin{claim}\label{clm:braid}
If $K_{m,n} (0) $ is Seifert fibered, then $K_{m,n} (0) - N^\circ(\widehat{F})$ is also Seifert fibered. 
\end{claim}
\begin{proof}
Let $M_{m,n}=K_{m,n}([\partial F]) =K_{m,n}(0)$. 
We can naturally extend the involution $\iota$ on $E(K_{m,n})$ to that on $M_{m,n}$, denoted by $\widehat{\iota}$, with the axis $\widehat{\alpha}$ appearing as the natural extension of $\alpha$. 
We may assume that $\widehat{F}$ is invariant under $\widehat{\iota}$. 
Then, as shown in \cite{Wu2006}, 
the torus appearing as $\partial N (\widehat{F})$ is incompressible in $M_{m,n}$. 
By \cite[VI.34. Theorem]{JacoBook}, such a torus is isotoped so that it is saturated in a Seifert fibration of $M_{m,n}$. 
That is, the $3$-manifold $M_{m,n} - N^\circ (\widehat{F})$ admits a Seifert fibration. 
\end{proof}

\begin{figure}[htb]
\includegraphics[width=0.95\textwidth]{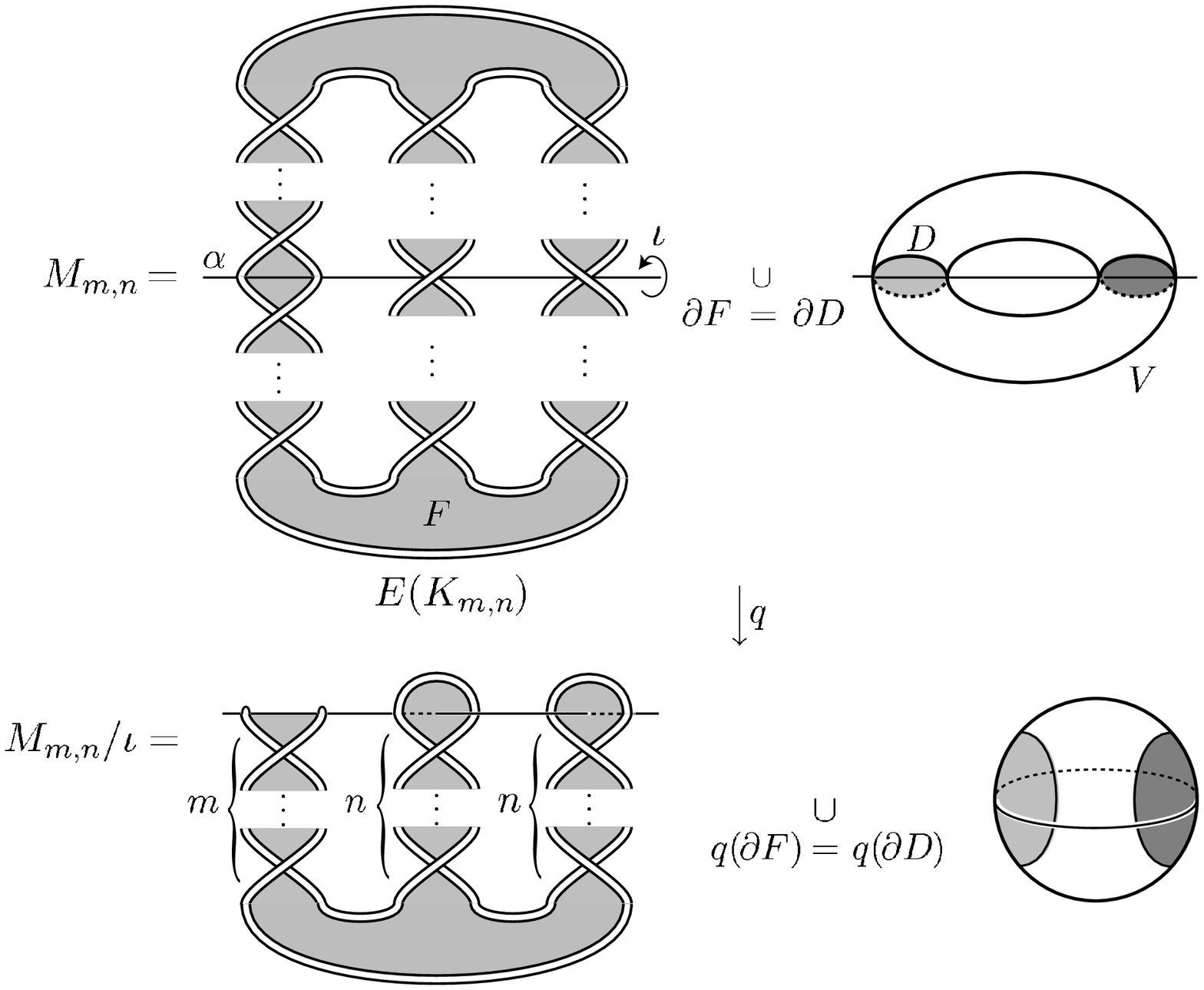}
\caption{}\label{fig:surgery1}
\end{figure}

Then $M_{m,n} / \widehat{\iota} = S^3$ is described as shown in the lower half of Figure~\ref{fig:surgery1}. 
Let $q: M_{m,n} \rightarrow S^3$ be the quotient map with respect to the involution $\widehat{\iota}$. 
Removing $N^\circ(q(\widehat{F}))$, that is, an open regular neighborhood of the pale shaded disk in the lower half of Figure~\ref{fig:surgery1}, and deforming by isotopy, we obtain the upper half of Figure~\ref{fig:remove}. 
In the upper half of Figure~\ref{fig:remove}, the arc in the boundary of the shaded disk lying on the boundary of the right $3$-ball is identified with the gray curve lying on the boundary of the left $3$-ball. 
Then the branch set with respect to the quotient map $q$ appears as two strings in a $3$-ball; see the lower half of Figure~\ref{fig:remove}. 
That is, the pair 
$( M_{m,n} / \widehat{\iota} - q(N^\circ(\widehat{F})) , (( M_{m,n}  - N^\circ(\widehat{F}) ) \cap \widehat{\alpha} ) / \widehat{\iota} \,)$ 
gives a two-string tangle $T_{m,n}$ in $S^3$ such that the double branched covering space for this tangle is homeomorphic to $M_{m,n} - N^\circ(\widehat{F})$. 

On the other hand, the tangle $T_{m,n}$ can be deformed as shown in Figures~\ref{fig:tangle1} and \ref{fig:tangle2}. 
Then taking the double branched covering branched along $T_{m,n}$, we obtain the exterior of a knot, denoted by $J_{m,n}$, in $S^2 \times S^1$; see Figure~\ref{fig:J}. 
From Figure~\ref{fig:J}, we see that the knot $J_{m,n}$ is the closed braid of $\beta_{m,n}$, where $\beta_{m,n} = \sigma_1 \cdots \sigma_n \sigma_{n+1}^{2m-1}\sigma_{n+2}^{-1} \cdots \sigma_{2n}^{-1} \in \mathbf{B}_{2n+1}(S^2)$ (see Figure~\ref{fig:braid}). 
Now we have the following. 

\begin{claim}\label{clm:J}
The $3$-manifold $K_{m,n}(0) - N^\circ(\widehat{F})$ is homeomorphic to the exterior $E(J_{m,n})$ of $J_{m,n}$ in $S^2 \times S^1$. 
\end{claim}

\begin{figure}[htb]
\includegraphics[width=0.95\textwidth]{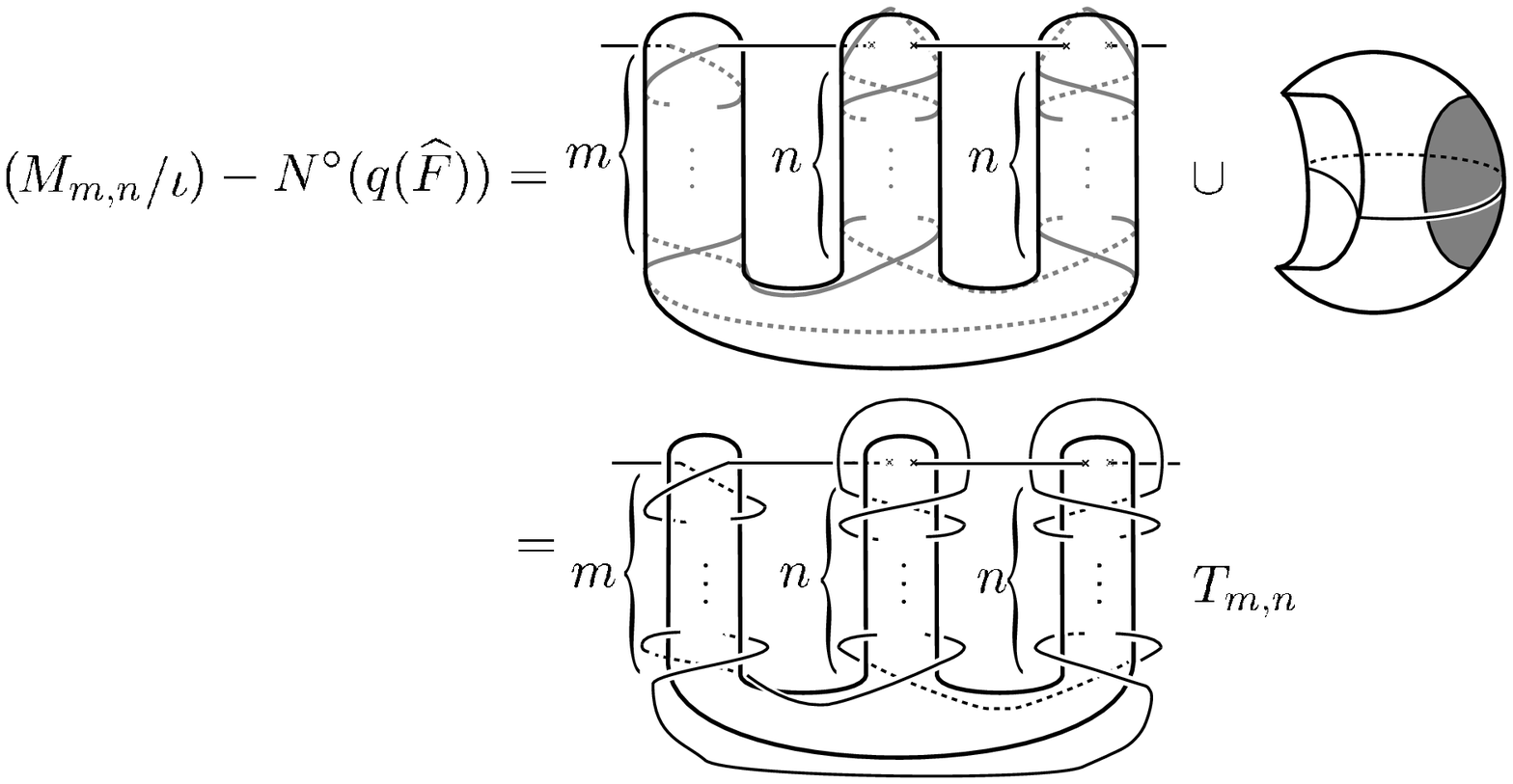}
\caption{}\label{fig:remove}
\end{figure}

\begin{figure}[htb]
\includegraphics[width=0.95\textwidth]{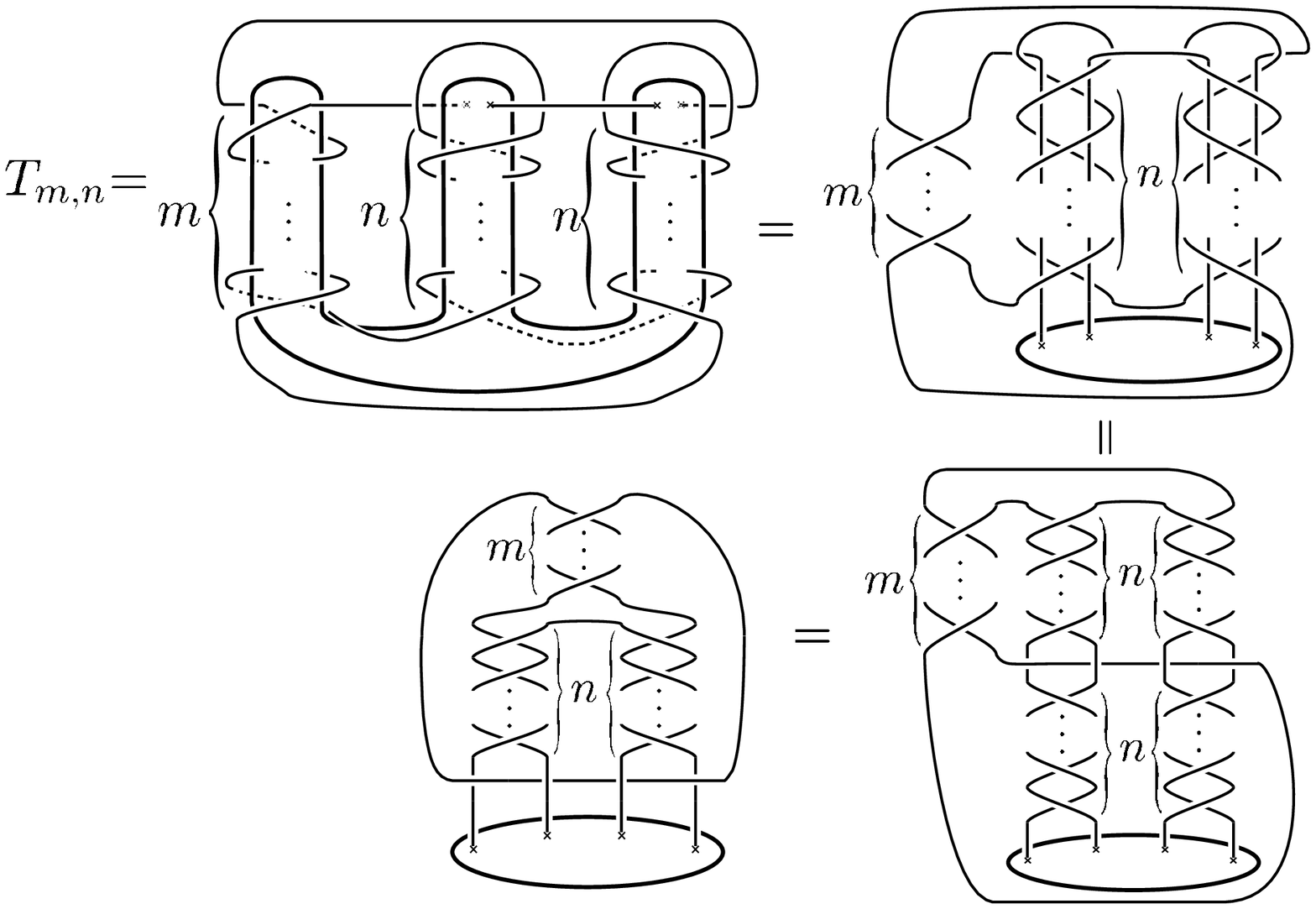}
\caption{}\label{fig:tangle1}
\end{figure}
\begin{figure}[htb]
\includegraphics[height=0.9\textheight]{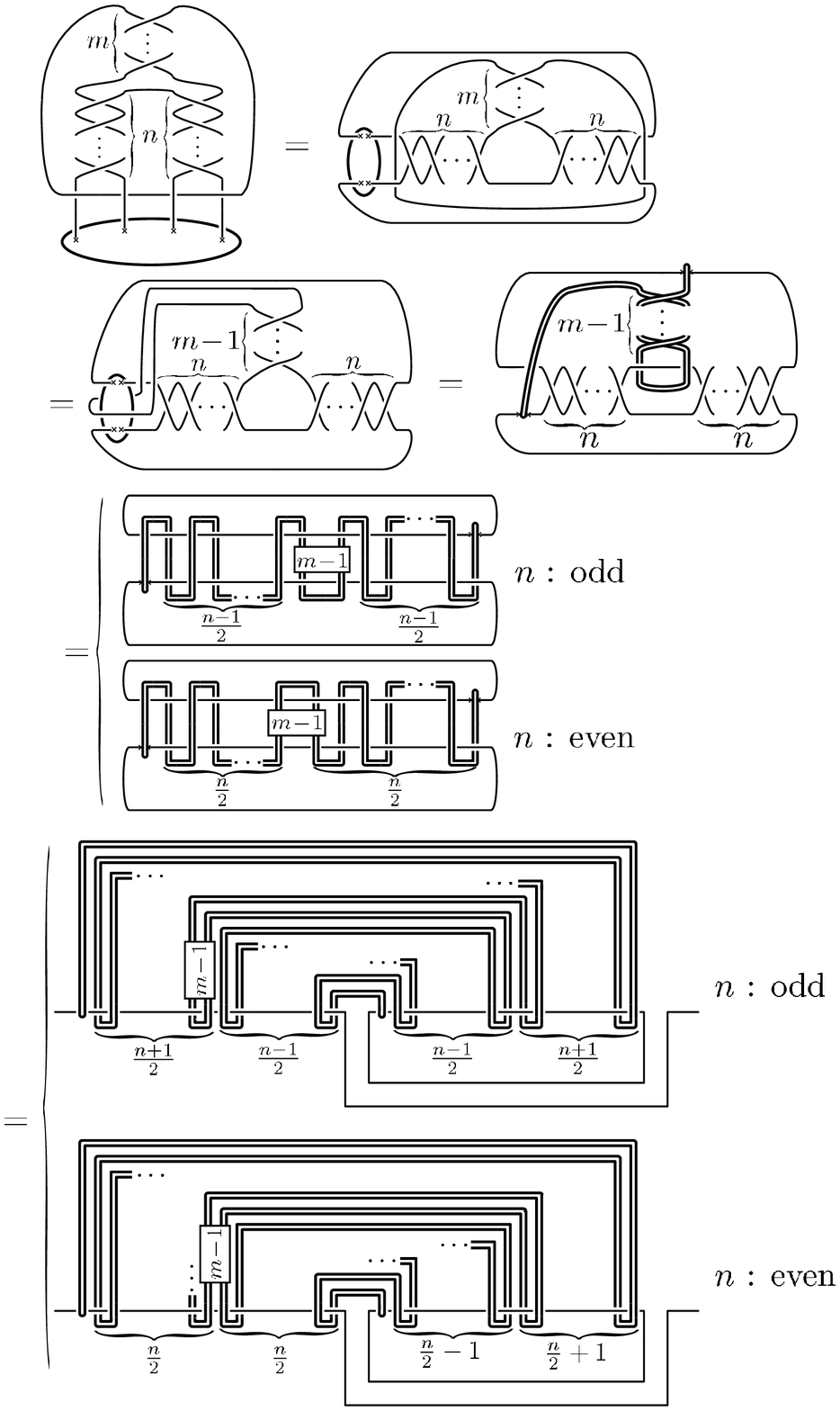}
\caption{}\label{fig:tangle2}
\end{figure}

\begin{figure}[htb]
\includegraphics[width=0.95\textwidth]{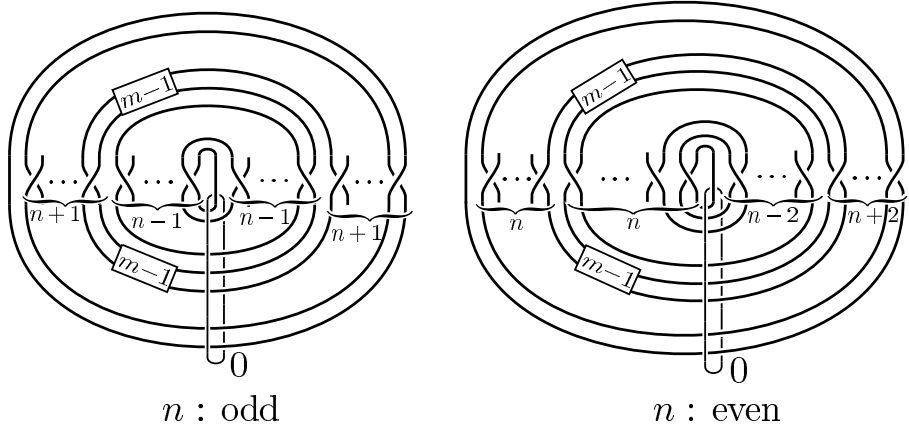}
\caption{$J_{m,n} = \widehat{\beta_{m,n}}$.}\label{fig:J}
\end{figure}
\begin{figure}[htb]
\includegraphics[width=0.2\textwidth]{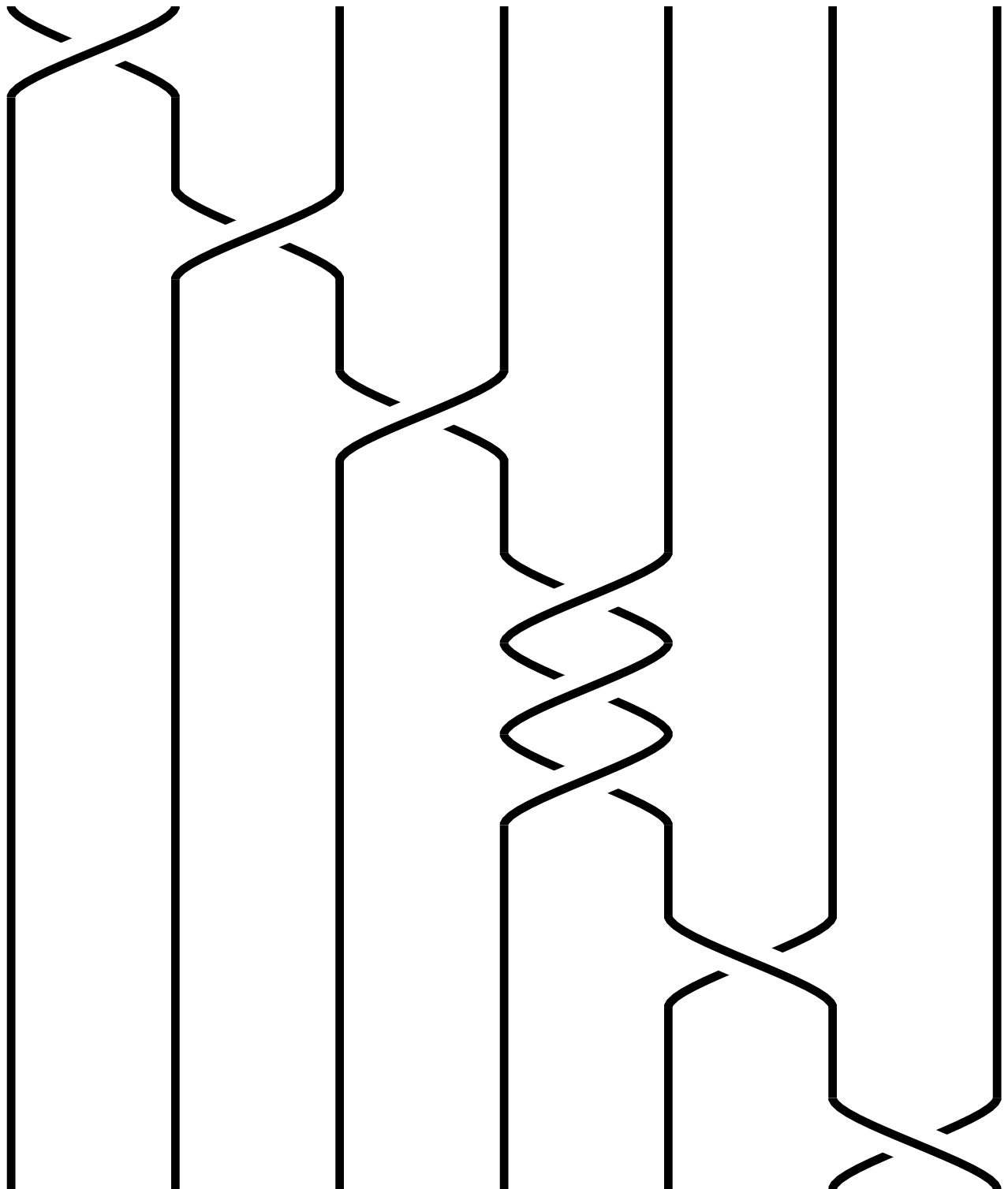}
\caption{$\beta_{2,3} = \sigma_1 \sigma_2 \sigma_3 \sigma_4^3 \sigma_5^{-1} \sigma_6^{-1}$.}\label{fig:braid}
\end{figure}

By Claims~\ref{clm:n=1}--\ref{clm:J}, it suffices to show that the knot $J_{m,n} \subset S^2 \times S^1$ 
has the exterior which is not Seifert fibered for $m \ne 0$, $n\ge2$. 

We start with the following claim.

\begin{claim}\label{clm:m=1}
The exterior $E(J_{\pm 1,n})$ of $J_{\pm 1,n}$ in $S^2 \times S^1$ is not Seifert fibered.
\end{claim}
\begin{proof}
It suffices to show that $E(J_{1,n})$ with $n\ge2$ is not Seifert fibered since $E(J_{-1,n})$ is orientation reversingly homeomorphic to $E(J_{1,n})$. 
We show that if $E(J_{1,n})$ is Seifert fibered, then $n=1$. 
Suppose that 
the exterior of $J_{1,n} \subset S^2 \times S^1$ admits a Seifert fibration. 
By \cite[Lemma 4]{Kohn1991}, this implies that $J_{1,n}$ is isotopic to a torus knot in $S^2 \times S^1$. 
Note that the exponent sum of a $(2n+1)$-string torus braid is congruent to $0$ or $2n$ modulo $4n$. 
Thus, by Lemma~\ref{lem:exp}, 
the exponent sum of $\beta_{1,n}$ must be congruent to $0$ or $2n$ modulo $4n$. 
On the other hand, we see that $\beta_{1,n}$ has the exponent sum $2$ and thus we have $n=1$. 
\end{proof}

Here we note that the exterior of $J_{1,n}$ is atoroidal since the exterior of $K_{m,n}$ admits only one punctured torus up to isotopy~\cite{Wu2006}. 
Also note that the exterior of $J_{1,n}$ is irreducible since it is a $(2n+1)$-punctured disk bundle over the circle. 
Therefore Claim~\ref{clm:m=1} implies that $J_{\pm 1,n}$ is hyperbolic for $n \ge 2$. 

We now consider the remaining case where $|m| \ge 2$ and $n \ge 2$ in the following two claims (Claims~\ref{clm:m>1} and \ref{clm:Ln-hyp}). 
To state these, let us consider 
the augmented two-component link $L_n = J_{1,n} \cup U$ in $S^2 \times S^1$ as shown in Figure~\ref{fig:Ln}. 
That is, $J_{1,n}$ is the knot considered in Claim~\ref{clm:m=1}, 
and $U$ is the trivial knot lying on a level sphere $S^2 \times \{ t \}$ such that the $-1/(m-1)$-surgery on $U$ yields the knot $J_{m,n}$ (see Figure~\ref{fig:Ln}). 

\begin{figure}[htb]
\includegraphics[width=0.45\textwidth]{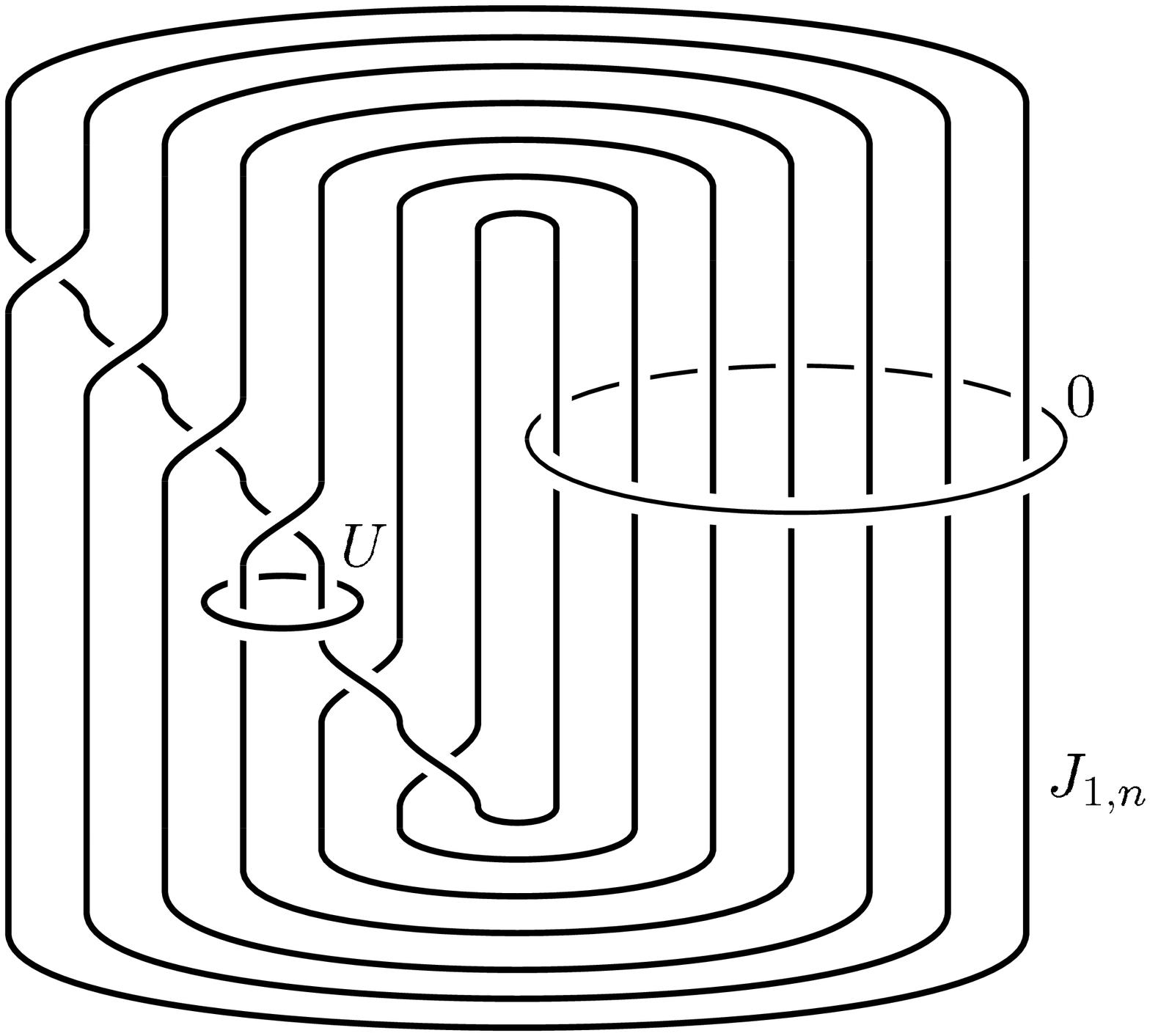}
\caption{The case where $n=3$; $L_3$.}\label{fig:Ln}
\end{figure}

\begin{claim}\label{clm:m>1}
If the link $L_n$ is hyperbolic, 
then $S^2 \times S^1 - N^\circ (J_{m,n})$ 
is not Seifert fibered for integers $|m| \ge 2$ and $n \ge 2$. 
\end{claim}
\begin{proof}
Suppose that the link $L_n$ is hyperbolic.
Assume for a contradiction that 
the exterior of $J_{m,n} \subset S^2 \times S^1$ is Seifert fibered for some $|m| \ge 2$ and $n\ge2$, and 
then $J_{-m,n}$ is also Seifert fibered. 
Notice that the knot $J_{m,n}$ is obtained from $J_{1,n}$ by twisting along $U$, 
equivalently, by the $-1/(m-1)$-surgery on $U$. 
This means that 
both of Dehn surgeries on $U$ along the slopes $- 1/ (m-1)$ and $- 1/ (-m-1)$ are Seifert fibered. 
By \cite[Corollary 1.2]{GordonWu2000}, the distance of such a pair is at most three since the link $L_n$ is hyperbolic. 
This contradicts the condition on the distance of the slopes $- 1/ (m-1)$ and $- 1/ (-m-1)$; 
$| (- 1)(-m-1) - (- 1)(m-1)| = | m+1 +m -1|=2 |m| \ge 4$. 
\end{proof}

\begin{claim}\label{clm:Ln-hyp}
The link $L_n$ is hyperbolic. 
\end{claim}
\begin{proof}
We show that $E(L_n)$ contains 
no essential sphere (i.e., $E(L_n)$ is irreducible), 
no essential disk (i.e., $E(L_n)$ is boundary-irreducible), 
no essential torus (i.e., $E(L_n)$ is atoroidal), 
and $E(L_n)$ is not Seifert fibered in Subclaims~\ref{subclm:1}--\ref{subclm:4} respectively. 
Then, by \cite{Thurston1982}, the link $L_n$ is hyperbolic. 

First we set the following notations used in the proof. 
Let $D_U$ be the disk bounded by $U$, which lies on a level sphere $S^2 \times \{ t \}$ and intersects $J_{1,n}$ transversely at just two points. 
Let $P$ be the twice-punctured disk appearing as the intersection $E(L_n) \cap D_U$, which is incompressible in $E(L_n)$ since it is an essential sub-surface of a fiber surface of $E(J_{1,n})$. 
There are only three isotopy classes of essential loops on $P$, 
two of which are parallel to the meridian of $J_{1,n}$, 
and the other one of which is parallel to $U$. 
For simplicity, we say that the former are {\it of type $(a)$} and the later is {\it of type $(b)$} (see Figure~\ref{fig:loop}). 
\begin{figure}[htb]
\includegraphics[width=0.6\textwidth]{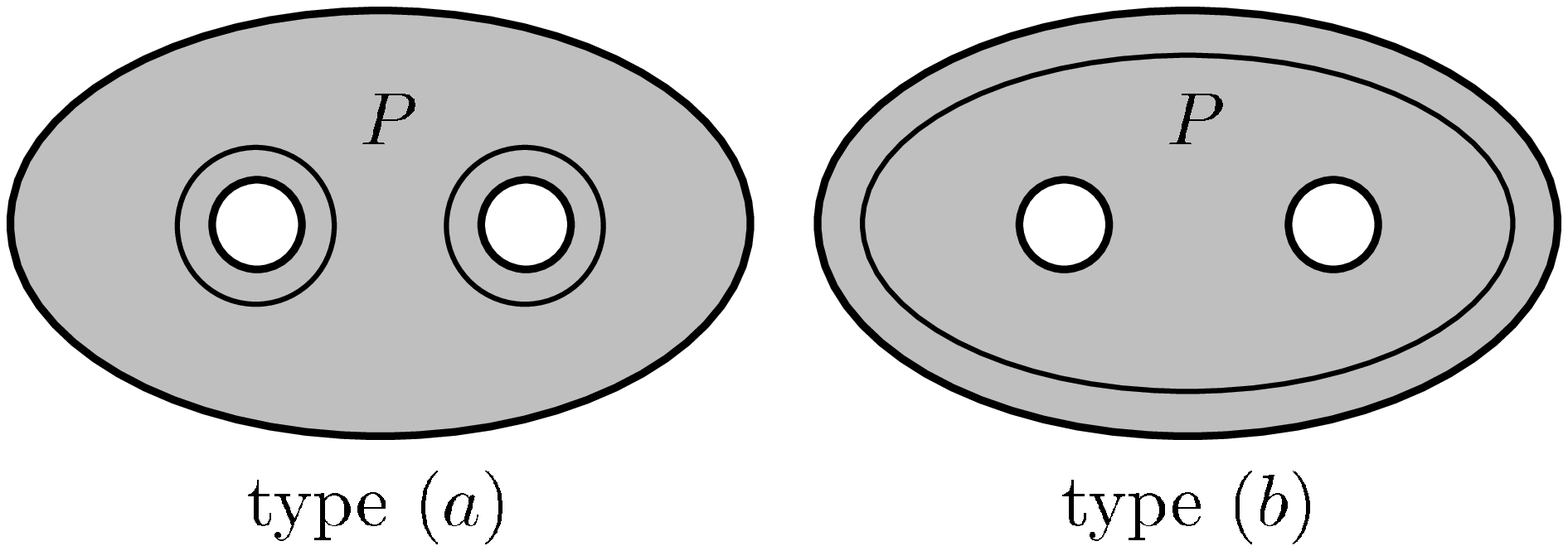}
\caption{}\label{fig:loop}
\end{figure}

The following subclaim is easily shown by Figure~\ref{fig:tunnel}. 
\setcounter{subclaim}{-1}
\begin{subclaim}\label{subclm:0}
The exterior $E(L_n \cup D_U)$ of $L_n \cup D_U$ in $S^2 \times S^1$ is homeomorphic to a handlebody of genus two. 
\end{subclaim}

\begin{figure}[htb]
\includegraphics[width=0.9\textwidth]{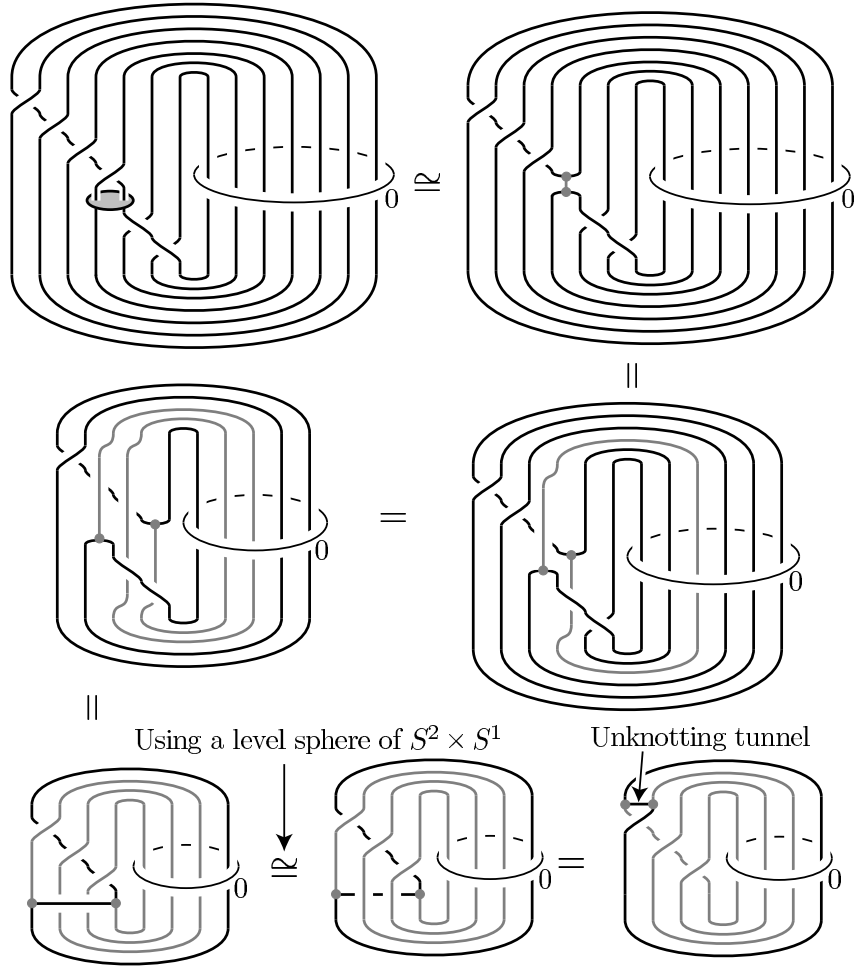}
\caption{The case where $n=3$. Other cases are shown in a similar way.}\label{fig:tunnel}
\end{figure}

We note that Subclaim~\ref{subclm:0} implies that $E(L_n \cup D_U)$ is irreducible and atoroidal.

\begin{subclaim}\label{subclm:1}
The exterior $E(L_n)$ is irreducible. 
\end{subclaim}
\begin{proof}
Assume for a contradiction that $E(L_n)$ contains an essential sphere $S$, that is, $S$ does not bound a $3$-ball in $E(L_n)$. 
Then $S$ must have non-empty intersection with $P$ since, if otherwise, $E(L_n \cup D_U)$ contains an essential sphere, which contradicts Subclaim~\ref{subclm:0}. 
Furthermore $S$ can be isotoped so that the intersection $S \cap P$ consists of only essential loops on $P$ by Subclaim~\ref{subclm:0}. 
Since each of the loops bounds a disk on $S$, it bounds a disk in $E(L_n)$. 
This contradicts that $P$ is incompressible in $E(L_n)$. 
\end{proof}

\begin{subclaim}\label{subclm:2}
The exterior $E(L_n)$ is boundary-irreducible. 
\end{subclaim}
\begin{proof}
Assume for a contradiction that $E(L_n)$ contains an essential disk $D$. 
Then, by compressing the torus component of $\partial E(L_n)$ along $D$, we have an embedded $2$-sphere in $E(L_n)$. 
By Subclaim~\ref{subclm:1}, the $2$-sphere bounds a $3$-ball, but 
this implies that $E(L_n)$ is homeomorphic to a $3$-ball with a $1$-handle, that is, a solid torus. 
A contradiction occurs. 
\end{proof}

\begin{subclaim}\label{subclm:3}
The exterior $E(L_n)$ is atoroidal. 
\end{subclaim}
\begin{proof}
Assume for a contradiction that $E(L_n)$ contains an essential torus $T$. 
Then $T$ must have non-empty intersection with $P$ since, if otherwise, $E(L_n \cup D_U)$ contains an essential torus, which contradicts Subclaim~\ref{subclm:0}. 
We isotope $T$ so that the intersection $T \cap P$ consists of loops which are essential in both $T$ and $P$. 
This is possible since $T$ and $P$ are incompressible in $E(L_n)$ together with Subclaim~\ref{subclm:1}. 
Then we have the following three cases:
\begin{enumerate}
\item[(i)]
$T \cap P$ includes both of loops of type $(a)$ and $(b)$. 
\item[(ii)]
$T \cap P$ consists of loops of type $(b)$ only. 
\item[(iii)]
$T \cap P$ consists of loops of type $(a)$ only. 
\end{enumerate}

\smallskip

\noindent
Case (i): 
In this case, there exists a loop $c_a$ of type $(a)$ and a loop $c_b$ of type $(b)$ in $T \cap P$, which cobounds an annulus $A_{a , b}$ on $T$, namely, $\partial A_{a , b} = c_a \cup c_b$. 
Then we have an embedded $2$-sphere $S_{a,b}$ in $S^2 \times S^1$, which is obtained from $A_{a,b}$ by capping of the boundaries by using the meridian disk bounded by $c_a$ and the sub-disk bounded by $c_b$ on $D_U$. 
Then $S_{a,b}$ intersects $J_{1,n}$ at just three points and thus $S_{a,b}$ is non-separating. 
This contradicts that $J_{1,n}$ is the closed $(2n+1)$-braid on $S^2$ and $n \ge 2$. 

\smallskip

\noindent
Case (ii): 
Recall that $E(J_{1,n})$ is atoroidal. 
Thus $T$ must be either boundary-parallel or compressible in $E(J_{1,n})$.

Suppose first that $T$ is boundary-parallel in $E(J_{1,n})$. 
Then $T$ bounds a solid torus $V_T$ in $S^2 \times S^1$ 
with $J_{1,n}$ as a core curve. 
Let $c_b$ be the component of $T \cap P$ 
which is innermost on $D_U$ among $T \cap P$. 
Then, since $T \cap P$ consists of loops of type $(b)$ only, 
$c_b$ bounds a disk in $D_U$ 
intersecting $J_{1,n}$ in two points. 
On the other hand, 
since $T$ bounds a solid torus $V_T$ in $S^2 \times S^1$ with $J_{1,n}$ as a core curve, 
there is a compressing disk $D_J$ of $T$ 
which intersects $J_{1,n}$ exactly once. 
Note that $c_b$ and $\partial D_J$ are isotopic on $T$, 
and hence 
they must have the same linking number 
with $J_{1,n}$ in $V_T$, a contradiction. 

Suppose next that $T$ is compressible in $E(J_{1,n})$, 
namely, there exists a compressing disk $D_T$ for $T$ in $E(J_{1,n})$. 
Note that $D_T \cap U \ne \emptyset$ 
since $T$ is incompressible in $E(L_n)$. 
Let $S_T$ be the $2$-sphere embedded in $E(J_{1,n})$ 
obtained from $T$ by compressing along $D_T$. 
Then, by irreducibility of $E(J_{1,n})$, 
there exists a $3$-ball $B_T$ bounded by $S_T$ in $E(J_{1,n})$. 
Let $Q_T$ be the compact sub-manifold of $E(J_{1,n})$ 
with $\partial Q_T = T$. 

If $B_T \supset D_T$, then $B_T$ is obtained from $Q_T$ by attaching a $2$-handle and thus $Q_T$ is obtained from $B_T$ by removing an open regular neighborhood of an arc properly embedded in $B_T$. 
Then $Q_T$ satisfies that 
$Q_T \cap D_T = \partial Q_T \cap \partial D_T = T \cap \partial D_T$, and also satisfies that $Q_T \cap U = \emptyset$ since $D_T \cap U \ne \emptyset$ and $T \cap U = \emptyset$. 
On the other hand, 
we find a sub-annulus $A_U$ on $D_U$ which connects $U$ and $T$ by the assumption that $T \cap P$ consists of loops of type $(b)$ only. 
Let $N_U$ be a regular neighborhood of $Q_T \cup D_T \cup A_U$ in $E(J_{1,n})$. 
Then $N_U$ is homeomorphic to a $3$-ball 
such that $U \subset N_U \subset E(J_{1,n})$. 
This means that there exists a 2-sphere $\partial N_U$ embedded in $E(L_n)$, separating the two boundary components. 
This contradicts Subclaim~\ref{subclm:1}. 

If $B_T \not\supset D_T$, 
then $Q_T$ is homeomorphic to a solid torus $V$ in $E(J_{1,n})$, 
which is obtained from $B_T$ by attaching a $1$-handle. 
Note that $D_T$ is a meridian disk in $V$. 
Also note that $U \subset V$ since $D_T \cap U \ne \emptyset$. 
Let $c$ be the loop in $T \cap P = T \cap D_U$ 
which is innermost on $D_U$. 
That is, $c$ bounds a disk $D$ on $D_U$ such that 
$D$ intersects $J_{1,n}$ at just two points 
and $D \cap T = \partial D \cap T = c$. 
Note that $c \subset T$ and 
$D \cap V = \partial D \cap \partial V = \partial D \cap T = c$. 
Then let us consider the intersection of $c$ and $\partial D_T$ on $T$. 

If $c$ is isotopic to $\partial D_T$ on $\partial V =T$, 
then there exists the $2$-sphere $S$ in $S^2 \times S^1$ consisting of $D_T$ and $D$, which intersects $J_{1,n}$ at just two points. 
Notice that the algebraic intersection number between $S$ and $J_{1,n}$ is two and thus $S$ is non-separating. 
This contradicts that $J_{1,n}$ is the closed $(2n+1)$-braid on $S^2$ and $n \ge 2$. 

If the slope of $c$ and that of $\partial D_T$ have 
distance greater than one, 
i.e., the minimal geometric intersection number of 
their representatives is greater than one 
on $\partial V =T$, 
then a regular neighborhood of $V \cup D$ gives a punctured lens space $L(p,q)$ with $q >1$ embedded in $S^2 \times S^1$, where $q$ is the distance between the slopes $c$ and $\partial D_T$. 
This contradicts that $S^2 \times S^1$ is prime. 

If the slopes of $c$ and $\partial D_T$ 
have distance just one on $\partial V =T$, then $U$ must be a core curve in $V$ as follows: 
Let $c' \subset T \cap P$ be the outermost loop on $D_U$. 
Since $c$ is parallel to $c'$ on $T$ and $c'$ is isotopic to $U$ in $V$, $c$ is isotopic to $U$ in $V$. 
This implies that $D_T$ is isotoped so that $U$ and $D_T$ intersects at a single point, meaning that $U$ is a core curve in $V$. 
Then $T = \partial V$ is boundary-parallel in $E(L_n)$, contradicting that $T$ is essential in $E(L_n)$. 

\smallskip

\noindent
Case (iii): 
In this case, there must exist an incompressible annulus $A_T$ embedded in $E(L_n)$, which connects $T$ and $\partial N(U)$. 
Assume for a contradiction that 
there are no incompressible annulus embedded in $E(L_n)$ connecting $T$ and $\partial N(U)$. 
Then, by~\cite[Theorem 1]{Wu1992}, 
$T$ survives incompressible with finitely many exceptions after Dehn surgeries on $U$. 
Since $(-1/(m-1))$-surgery on $U$ gives the exterior of the knot $J_{m,n}$, 
which is atoroidal for $m \ne 0$ and $n\ge2$, 
the torus $T$ must be boundary-parallel in $E(J_{m,n})$ 
after all but finitely many Dehn surgeries on $U$. 
This implies that $T$ is separating in $E(J_{1,n})$, 
and $U$ is contained in the component of the complement of $T$ containing $\partial E(J_{1,n})$. 
Now cut $E(J_{1,n})$ along $T$ to obtain a 3-manifold $N$ 
having two torus boundaries which correspond to $\partial N(J_{1,n})$ and $T$. 
For brevity, we also denote the two boundaries of $N$ by the same symbol $\partial N(J_{1,n})$ and $T$ respectively. 
These boundary tori are all incompressible in $N$ since they are all incompressible in $E(J_{1,n})$. 
Now, by the assumption of the case, 
there is an annulus embedded in $N$ connecting $\partial N(J_{1,n})$ and $T$. 
Then by \cite[Theorem 2.4.3 (b)]{CGLS1987}, after some Dehn filling on $\partial N(J_{1,n})$, 
i.e., attaching a solid torus along $\partial N(J_{1,n})$, 
the boundary $T$ must be incompressible in the resultant manifold, say $N'$. 
On the other hand, by the assumption, 
$T$ must become parallel to $\partial N(J_{1,n})$ after 
all but finitely many Dehn surgeries on $U$. 
Thus $T$ must become compressible 
in the manifolds obtained from $N'$ 
by all but finitely many Dehn surgeries on $U$. 
It then follows from \cite[Theorem 1]{Wu1992} again that 
there exists an incompressible annulus in $N \subset E(L_n)$ 
which connects $T$ and $\partial N(U)$. 
This contradicts the assumption. 

Furthermore since $T$ becomes inessential after any $(-1/(m-1))$-surgery for $m \ne 0 \in \mathbb{Z}$, the boundary slope of $A_T$ on $\partial N(U)$ has distance one from any slope represented by $-1/(m-1)$~\cite[Theorem 2.4.3]{CGLS1987}. 
This means that the boundary slope of $A_T$ on $\partial N(U)$ must be $0/1$, namely, longitudinal.

Let $c_A$ be the loop 
$A_T \cap T$. 
If $c_A$ is trivial on $T$, then the disk bounded by $c_A$ with $A_T$ gives 
an essential disk in $E(L_n)$, contradicting to Subclaim~\ref{subclm:2}. 
Hence $c_A$ is non-trivial on $T$. 
Now we further isotope $T$ so that 
the intersection $T \cap P$ has the minimal number of components. 

Let us consider the intersection $c_A \cap P$. 
We may assume that $c_A$ and $P$ intersect transversely. 
Furthermore we can isotope $A_T$ so that $c_A \cap P = \emptyset$ as follows. 
Assume that $c_A \cap P \ne \emptyset$. 
Since the boundary slope of $A_T$ on $\partial N(U)$ is longitudinal, 
$A_T$ is isotoped near $\partial N(U)$ so that $( A_T \cap \partial N(U) ) \cap ( P \cap \partial N(U) ) = \emptyset$. 
Then there exists an arc $\alpha_A \subset A_T \cap P$ on $P$ such that $\partial \alpha_A \subset T \cap P$. 
Notice that $\alpha_A$ must be parallel to a sub-arc in $T \cap P$. 
For otherwise, 
$T$ is isotoped so that the number of components $T \cap P$ is reduced
by using the bigon on $A_T$ bounded by $\alpha_A$ 
with a sub-arc on $\partial A_T$; see Figure~\ref{fig:annulus}. 
This contradicts the assumption. 
Then, by the incompressibility of $T$ and Subclaim~\ref{subclm:1}, 
we can isotope $A_T$ so that $\alpha_A$ vanishes. 
Repeating this process, 
$A_T$ is isotoped so that $c_A \cap P = \emptyset$. 

Since $c_A$ is non-trivial and $c_A \cap P = \emptyset$, together with the assumption that $T \cap P$ consists of loops of type $(a)$ only, 
$c_A$ is parallel to the meridian of $J_{1,n}$. 

Let us consider the intersection $A_T \cap P$. 
If $A_T \cap P = \emptyset$, then as in Case (i), 
we would find a $2$-sphere embedded in $S^2 \times S^1$, 
which intersects $J_{1,n}$ at just three points. 
This contradicts that $J_{1,n}$ is the closed $(2n+1)$-braid on $S^2$ and $n \ge 2$. 

Now we consider the case where $A_T \cap P \ne \emptyset$. 
Then $A_T \cap P$ consists of loops on $P$ since all arc components vanish by the above argument. 
Since $A_T$ and $P$ are both incompressible in $E(L_n)$, 
together with Subclaim~\ref{subclm:1}, the loops in $A_T \cap P$ are essential on both $A_T$ and $P$. 
This implies that all of the loops in $A_T \cap P$ are parallel on $A_T$. 
If the loop in $A_T \cap P$ nearest to $c_A$ on $A_T$ is of type $(b)$ on $P$, 
then, as in the case where $A_T \cap P = \emptyset$, we have a contradiction. 
If the loop in $A_T \cap P$ nearest to $\partial A_T \cap \partial N(U)$ is of type $(a)$ on $P$, then, we also have a contradiction in the same way. 
Otherwise, there must be a pair of loops in $A_T \cap P$ cobounding a sub-annulus on $A_T$, one of which is of type $(a)$ on $P$ and the other is of type $(b)$ on $P$. 
Again, in this case, we have a contradiction in the same way. 

Now we complete the proof of Subclaim~\ref{subclm:3}.  
\end{proof}

\begin{figure}[htb]
\includegraphics[width=0.7\textwidth]{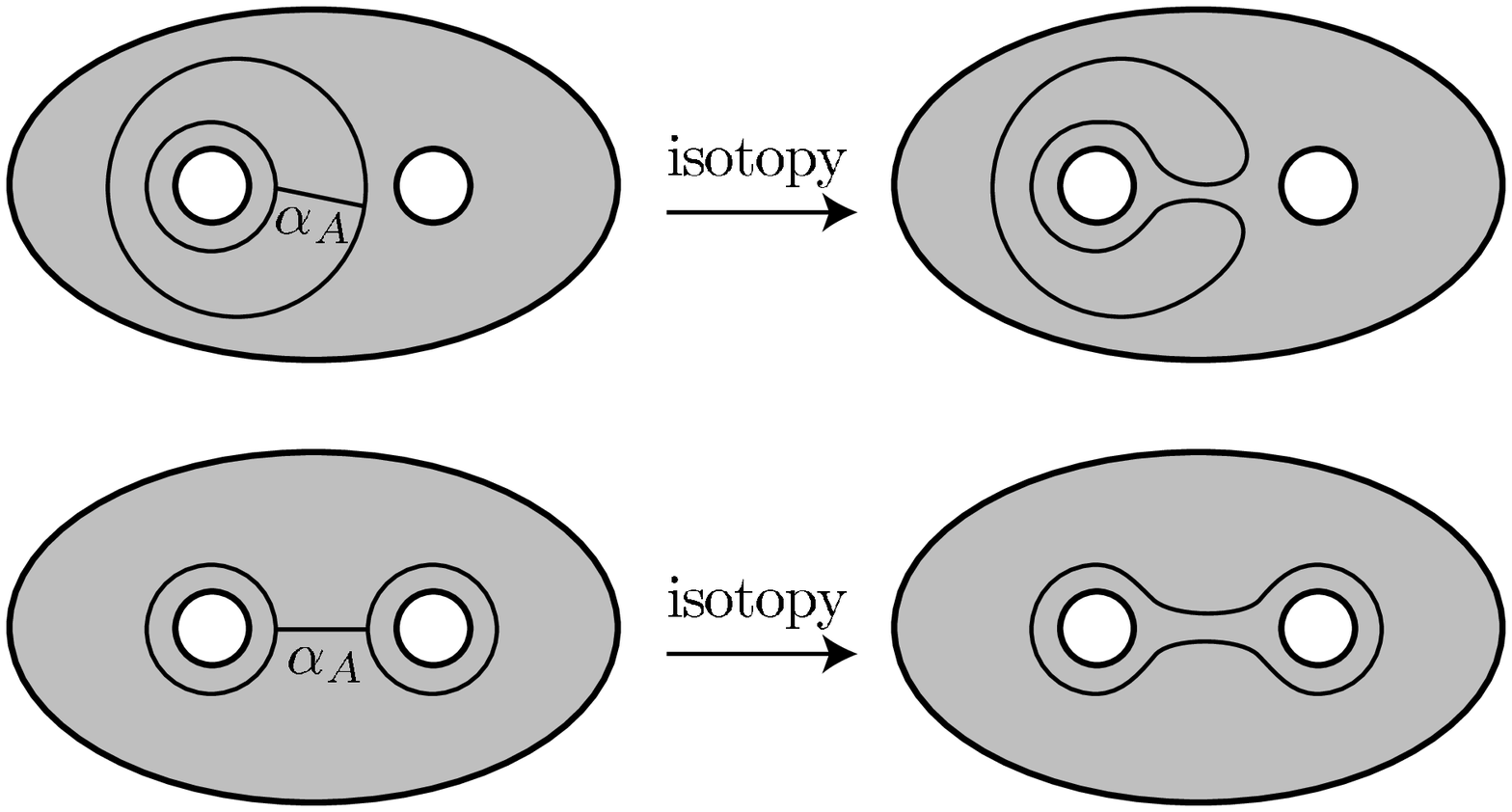}
\caption{}\label{fig:annulus}
\end{figure}

\begin{subclaim}\label{subclm:4}
The exterior $E(L_n)$ is not Seifert fibered. 
\end{subclaim}
\begin{proof}
If $E(L_n)$ is Seifert fibered, then no Dehn surgery on $U$ gives a hyperbolic manifold. 
However, the trivial surgery on $U$ yields the exterior of $J_{1,n}$, which is hyperbolic. 
\end{proof}

By Subclaims~\ref{subclm:1}--\ref{subclm:4}, we complete the proof of Claim~\ref{clm:Ln-hyp}. 
\end{proof}

Consequently, by Claims~\ref{clm:n=1}, \ref{clm:m=1}--\ref{clm:Ln-hyp}, we complete the proof of Lemma~\ref{lem:pretzel}. 
\end{proof}

\section{On non-pretzel Montesinos knots}\label{sec:Monte}

In this section, we show the following:

\begin{lemma}
For the Montesinos knot $K_\pm=M(-1/2, 1/3, 1/(6 \pm 1/2) )$, 
the surgered manifold $K_\pm(0)$ is not Seifert fibered. 
\end{lemma}
\begin{proof}
We divide our argument into two claims.

\begin{claim}\label{clm:Monte+}
The surgered manifold $K_+(0)$ is not Seifert fibered. 
\end{claim}
\begin{proof}
As mentioned in Remark~\ref{rem:fibered}, the knot $K_+$ is fibered. 
Let $F$ be a fiber surface for $E(K_+)$, and 
$h : F \rightarrow F$ a monodromy homeomorphism with respect to the fibration of $E(K_+)$, that is, $E(K_+) \cong F \times [ 0,1 ] / (x,1) \sim (h(x),0)$. 
By the natural extension, the surgered manifold $K_+(0)$ is regarded as 
a surface bundle with the monodromy map $\hat{h}: \hat{F} \rightarrow \hat{F}$, 
where $\hat{F}$ denotes the closed surface obtained from $F$ by capping off. 
Consider the induced isomorphism $\hat{h}_* : H_1(\hat{F}) \rightarrow H_1(\hat{F})$ 
from the monodromy $\hat{h}$. 
Notice that the characteristic polynomial of $\hat{h}_*$ 
coincides with the Alexander polynomial of $K_+$ 
(see for example~\cite[p.~326, Corollary 8]{RolfsenBook}). 
On the other hand, 
if the characteristic polynomial of $\hat{h}_*$ has no roots of unity as zeros, 
then the  homeomorphism $h$ is not periodic~\cite[Lemma 5.1]{CassonBleiler1988}. 
These imply that 
if $\Delta_{K_+}(t)$ has no root of unity as zeros, 
then $K_+(0)$ is not Seifert fibered. 
Actually, we have $\Delta_{K_+}(t) = -(t^2+t^{-2}) + 3$, 
and then $\Delta_{K_+}(t)$ has no roots of unity as zeros. 
This completes the proof of Claim~\ref{clm:Monte+}. 
\end{proof}

In the remaining part, we will consider Dehn surgeries on links in $S^3$, 
equivalently, surgery descriptions for $3$-manifolds. 
We omit the details. 
See \cite[Chapter 9]{RolfsenBook} for example.

\begin{claim}\label{clm:Monte-}
The surgered manifold $K_-(0)$ is not Seifert fibered. 
\end{claim}

\begin{proof}
As illustrated in Figure~\ref{fig:Kirbymove}, the knot $K_-$ with the surgery coefficient $0$ is 
transformed into the link $L$ with the surgery coefficients $\mathbf{r} =(0,1,1,-1,-1,-1/3)$ by the Kirby moves. 

Then, as seen in \cite[section 3]{Ichihara1998}, 
the $3$-manifold $L(\mathbf{r})$ obtained by the Dehn surgery 
on $L$ with surgery coefficients $\mathbf{r}$ is homeomorphic to 
a surface bundle $M_\varphi$ over the circle with genus two surface fibers, 
and its monodromy map $\varphi$ is described as 
$$\varphi = \tau_5^3 \circ \tau_3 \circ \tau_1^{-1} \circ \tau_4 \circ \tau_2^{-1} \ .$$
Here $\tau_i$ ($1 \le i \le 5$) denotes the Dehn twist on the genus two surface 
along the simple closed curve $c_i$ depicted in Figure~\ref{fig:LickorishGenerators}. 

Furthermore, 
since $\varphi$ commutes with the hyper-elliptic involution, 
the manifold $L(\mathbf{r})$ is also obtained 
as the double branched covering space of $S^2 \times S^1$ branched along the $6$-string closed braid $\hat{\beta}$ in $S^2 \times S^1$ with 
$$\beta = \sigma_5^3 \sigma_3 \sigma_1^{-1} \sigma_4 \sigma_2^{-1} \in \mathbf{B}_6 (S^2). $$

Now, assume for a contradiction that $K_-(0) \cong M_\varphi$ is Seifert fibered. 
Then, since the fiber surface is of genus two, 
by \cite[VI. 31. Lemma]{JacoBook}, the monodromy map $\varphi$ must be periodic. 
This implies that the braid $\beta$ is of finite order in $\mathbf{B}_6 (S^2)$. 
Then, by \cite[Theorem 4.2]{Murasugi1982}, 
$\beta$ is conjugate to $\Delta^k$ for some integer $k$, where 
$\Delta$ denotes $\sigma_1 \sigma_2 \sigma_3 \sigma_4 \sigma_5$. 
This implies that the closed braid for $\beta$ is isotopic to that for $\Delta^k$. 
%
However, the exponent sum of $\beta$ is 3, and that of $\Delta^m$ is $5m$.
Since $3 \not\equiv 5m \mod 10$, this contradicts Lemma~\ref{lem:exp}. 
This completes the proof of Claim~\ref{clm:Monte-}. 
\end{proof}
\begin{figure}[htb]
\includegraphics[width=0.6\textwidth]{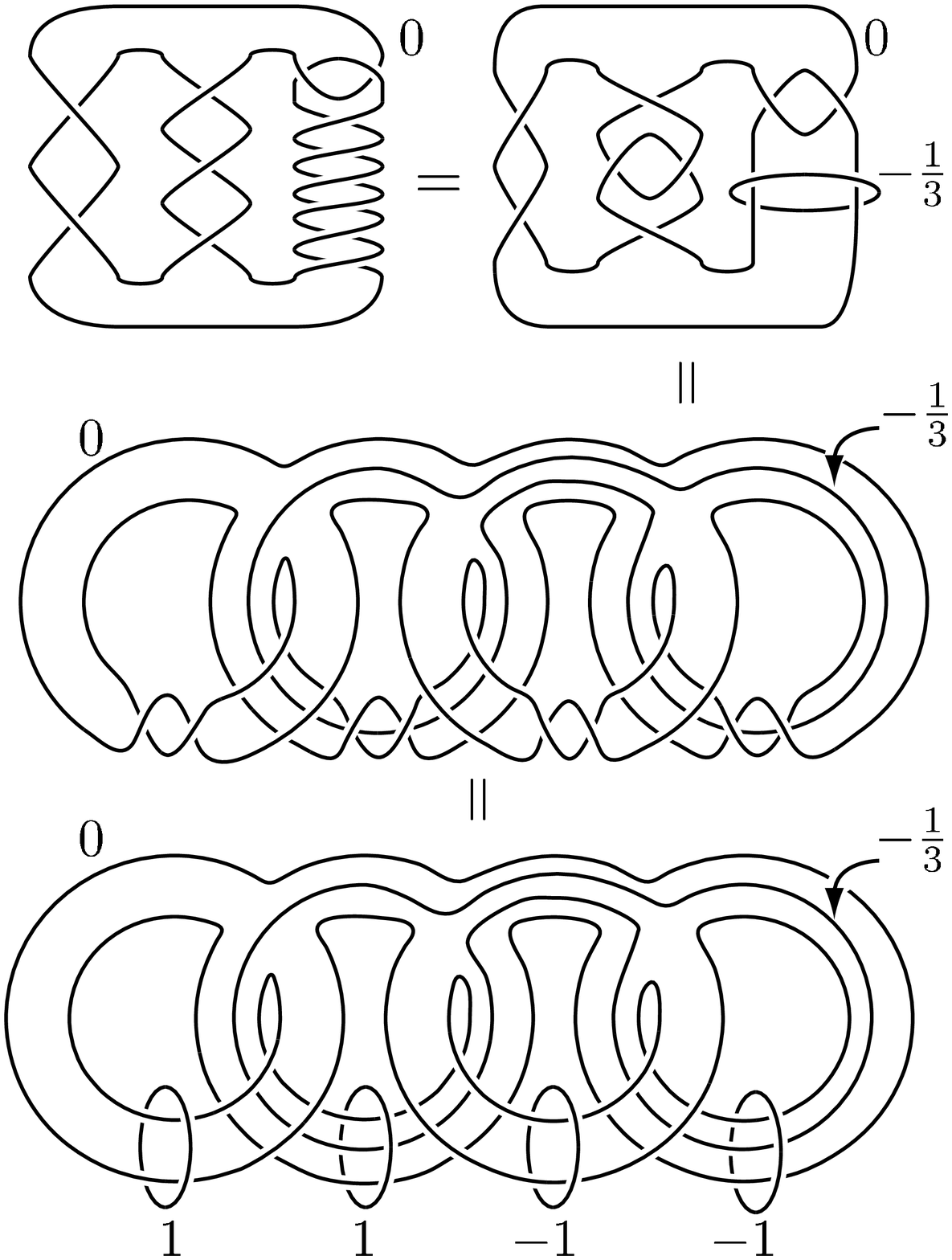}
\caption{}\label{fig:Kirbymove}
\end{figure}
\begin{figure}[htb]
\includegraphics[width=0.3\textwidth]{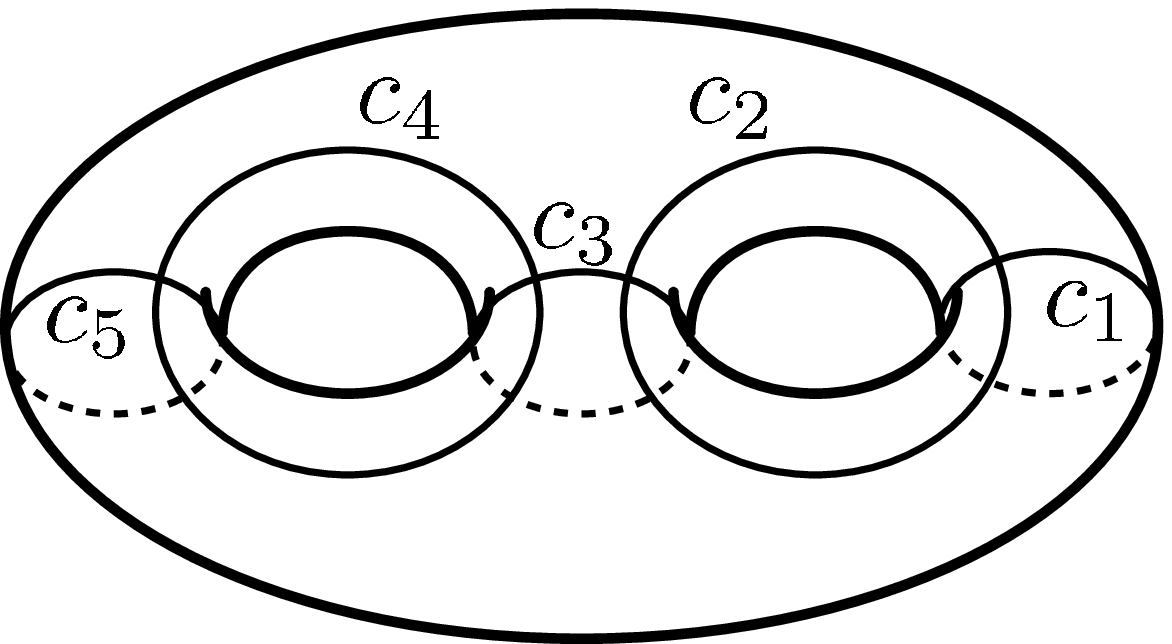}
\caption{}\label{fig:LickorishGenerators}
\end{figure}

By Claims~\ref{clm:Monte+} and \ref{clm:Monte-}, we complete the proof of Lemma~\ref{lem:non-pretzel}.
\end{proof}

\begin{remark}
Unfortunately the argument in the proof of Claim~\ref{clm:Monte+} 
(resp. Claim~\ref{clm:Monte-}) 
does not work for proving Claim~\ref{clm:Monte-} (resp. Claim~\ref{clm:Monte+}). 
Thus we divide the proof into two cases 
according to their signs, and apply different methods. 
\end{remark}

\subsection*{Acknowledgments}
The authors would like to thank Professor Makoto Sakuma for 
letting them know the result obtained by Professor Kunio Murasugi in \cite{Murasugi1982}. 
They also thank the referees for their suggestions and comments, 
in particular, on improving the proofs of Claims~\ref{clm:n=1} and Subclaim~\ref{subclm:3}.

\providecommand{\bysame}{\leavevmode\hbox to3em{\hrulefill}\thinspace}
\providecommand{\MR}{\relax\ifhmode\unskip\space\fi MR }
\providecommand{\MRhref}[2]{%
  \href{http://www.ams.org/mathscinet-getitem?mr=#1}{#2}
}
\providecommand{\href}[2]{#2}


\end{document}